%% file: massey.tex
\documentclass{amsart}

\usepackage{dha}

\author{Fernando Muro}
\address{Universidad de Sevilla,
Facultad de Matemáticas,
Departamento de Álgebra,
Avda. Reina Mercedes s/n,
41012 Sevilla, Spain}
\email{fmuro@us.es}
\urladdr{http://personal.us.es/fmuro}

\title{Massey products for algebras over operads}

\begin{document}

\begin{abstract}
  We define a generalization of Massey products for algebras over a Koszul operad in characteristic zero, extending Massey's and Allday's and Retah's in the associative and Lie cases, respectively. We establish connections with minimal models and with Dimitrova's universal operadic cohomology class. We compute a Gerstenhaber algebra example and a hypercommutative algebra example
  related to the Chevalley--Eilenberg complex of the Heisenberg Lie algebra.
\end{abstract}

\thanks{The author was partially supported by the grants PID2020-117971GB-C21
  funded by MCIN/AEI/10.13039/501100011033, US-1263032 (US/JUNTA/FEDER, UE), and
  P20\_01109 (JUNTA/FEDER, UE). He thanks Luis Narváez for conversations leading to the hypercommutative example, José Manuel Moreno-Fernández for suggesting some improvements and providing references, and Marco Farinati for interesting conversations on Lie bialgebras.}



\maketitle

\tableofcontents

\section{Introduction}

A Massey product
$\langle a,b,c\rangle$
is a well-known degree $1$ secondary operation in the homology $H_*(A)$ of a DG-associative algebra $A$ which is defined whenever $ab=0=bc$. It is an element in the following quotient
\[\langle a,b,c\rangle\in \frac{H_{\abs{a}+\abs{b}+\abs{c}+1}(A)}{H_{\abs{a}+\abs{b}+1}(A)c+aH_{\abs{b}+\abs{c}+1}(A)}.\]
It is often regarded as a subset
\[\langle a,b,c\rangle\subset H_{\abs{a}+\abs{b}+\abs{c}+1}(A)\]
and the denominator of the previous quotient is considered its indeterminacy.

This operation was introduced by Massey in \cite{massey_1969_higher_order_linking}. It is non-trivial in the cohomology of the complement of the Borromean link in $S^3$, which has trivial cup-product. It is closely related to the associativity relation
\[(ab)c-a(bc)=0\]
since it is defined as follows. If $\alpha,\beta,\gamma$ are cycles representing $a,b,c$ and $\zeta,\xi$ are chains satisfying $d(\zeta)=\alpha\beta, d(\xi)=\beta\gamma$ then
\[\zeta\gamma-(-1)^{\abs{\alpha}}\alpha\xi\]
is a cycle because its differential yields the associativity relation. The homology class of this cycle represents $\langle a,b,c\rangle$.

Allday \cite{allday_1973_rational_whitehead_products,allday_1977_rational_whitehead_products}, and independently Retah \cite{retah_1977_massey_operations_lie,retah_1978_massey_operations_lie}, introduced a similar operation in the homology $H_*(L)$ of a DG-Lie algebra $L$, called Lie-Massey products,
\[\langle a,b,c\rangle\in \frac{H_{\abs{a}+\abs{b}+\abs{c}+1}(L)}{[H_{\abs{a}+\abs{b}+1}(L),c]+[a,H_{\abs{b}+\abs{c}+1}(L)]}.\]
They used it in applications to rational homotopy theory.
This operation is related to the graded Jacobi identity
\[[a,[b,c]]-[[a,b],c]-(-1)^{\abs{a}\abs{b}}[b,[a,c]]=0\]
in the same way as the previous one is connected to the associativity relation.

Given a graded quadratic Koszul operad $\O{O}$ over a field $\Bbbk$ of characteristic zero, we here define a secondary Massey-product-like operation in the homology $H_*(A)$ of a DG-$\O{O}$-algebra $A$ for each relation in the presentation of $\O{O}$. This extends the original Massey and the Lie-Massey products when $\O{O}$ is the associative and the Lie operad, respectively. We show that they are connected to the higher operations in a minimal transferred $\O{O}_\infty$-algebra structure on $H_*(A)$. In particular, they yield obstructions to formality. It is worth mentioning at this point that this connection was already known for the associative operad. Moreover, it was believed to extend to higher Massey products in the cohomology of a differential graded associative algebra. This extension is actually stated as \cite[Theorem 3.1]{lu_palmieri_wu_zhang_2009_ainfinity_structure_extalgebras}. Inspired by that, higher operations in $\O{O}_\infty$-algebras were called Massey products in \cite{galvez-carrillo_tonks_vallette_2012_homotopy_batalinvilkovisky_algebras,loday_vallette_2012_algebraic_operads}. However, a counterexample to the aforementioned belief (and theorem) was recently discovered by \cite{buijs_moreno-fernandez_murillo_2020_infty_structures_massey}. This justifies the study of Massey products on their own.

We also establish a connection between our Massey-product-like operations and Dimitrova's \cite{dimitrova_2012_obstruction_theory_operadic} universal class of an operadic algebra $A$, which is another obstruction to formality.

We illustrate this new definition with Gerstenhaber and hypercommutative algebras. Gerstenhaber algebras are commutative algebras equipped with a shifted Lie bracket satisfying the following compatibility relation, called \emph{Gerstenhaber relation},
\[[a,b c]=[a,b] c+(-1)^{(\abs{a}-1)\abs{b}}b[a,c].\]
Hypercommutative algebras are also commutative algebras, now equipped with a whole sequence of operations of arities $r\geq 3$,
\[(a_1,\dots,a_r).\]
They satisfy several relations, including
\begin{equation*}
  (a b, c, x)+(-1)^{\abs{c}\abs{x}}(a,b,x) c = a (b,c,x)+(a,b c,x).
\end{equation*}
These kinds of algebras arise in different but related contexts, like differential geometry \cite{koszul_1985_crochet_schoutennijenhuis_cohomologie,dotsenko_shadrin_vallette_2015_rham_cohomology_homotopy} and Lie algebra cohomology \cite{kosmann-schwarzbach_1995_exact_gerstenhaber_algebras, Campos2016}. In the last section, we compute non-vanishing Massey products associated to these relations in the latter context and we relate them to the former.

In general, DG-algebras over a Koszul operad $\O{O}$ arise naturally as the operadic cochain complex of an algebra over the Koszul dual operad $\O{O}^{!}$. This construction yields unlimited examples to test the existence of non-trivial Massey products.
We aim at keeping this paper short, but several related topics could be addressed next. In the associative and Lie settings, higher order Massey products are defined whenever shorter ones vanish, and they play a role in several applications, see e.g.~\cite{allday_1973_rational_whitehead_products,allday_1977_rational_whitehead_products,retah_1977_massey_operations_lie,retah_1978_massey_operations_lie}. As in the associative case, we expect them to be not so related to higher operations in minimal models, compare \cite{buijs_moreno-fernandez_murillo_2020_infty_structures_massey} mentioned above. Hence their study would be of independent interest. It could also be worth to extend Massey products to algebras over generalizations of operads, such as colored operads, cyclic operads, etc. This should fit the Massey products in the homology of an operad used in \cite{livernet_2015_nonformality_swisscheese_operad} to prove that the Swiss cheese operad is not formal.

We here nominally work with Koszul operads in characteristic zero \cite{ginzburg_kapranov_1994_koszul_duality_operads,loday_vallette_2012_algebraic_operads}. We borrow terminology and notation from \cite{loday_vallette_2012_algebraic_operads}. In Remark \ref{hypotheses} we indicate how most things make sense in positive characteristic, or even over general ground rings and non-Koszul operads. This is relevant for applications of Massey products to minimal models for operadic algebras over commutative rings, extending \cite{sagave_2010_dgalgebras_derived_algebras}, see \cite{Maes2021,Muro2021}.

The degree of a homogeneous element $x\in X$ in a graded module $X$ is denoted by $\abs{x}$.


\section{Massey products}\label{firstsection}

\input{sections/massey.tex}

\section{Massey products and minimal models}\label{mm}

\input{sections/massey_minimal.tex}

\section{Connection with Dimitrova's universal class}\label{ump}

\input{sections/universal.tex}

\section{Examples}

\input{sections/application.tex}

\providecommand{\bysame}{\leavevmode\hbox to3em{\hrulefill}\thinspace}
\providecommand{\MR}{\relax\ifhmode\unskip\space\fi MR }
\providecommand{\MRhref}[2]{%
  \href{http://www.ams.org/mathscinet-getitem?mr=#1}{#2}
}
\providecommand{\href}[2]{#2}

\end{document}

%% file: sections/massey.tex
Let $\O{O}=\P{E}{R}=\F(E)/(R)$ be a quadratic Koszul graded operad generated by a graded reduced $\mathbb{S}$-module $E$ with sub-$\mathbb{S}$-module of relations $R\subset\F(E)^{(2)}$.

\begin{definition}\label{massey_product}
    Let
    \begin{equation}\label{relation}
        \Gamma=\sum(\mu^{(1)}\circ_l\mu^{(2)})\cdot\sigma\in R(r)
    \end{equation}
    be a \emph{relation} of arity $r$. In this Sweedler-like notation, $\mu^{(i)}\in E(r_i)$, $r_1+r_2-1=r$,  $1\leq l\leq r_1$, $\sigma\in\mathbb{S}_r$, and $\circ_l$ denotes the infinitesimal composition.

    Let $A$ be a DG-$\O{O}$-algebra and let $x_i\in H_{*}(A)$, $1\leq i\leq r$, be elements such that
    \begin{equation}\label{vanishing}
        \mu^{(2)}(x_{\sigma^{-1}(l)},\dots, x_{\sigma^{-1}(l+r_2-1)})=0\in H_{\abs{\mu^{(2)}}+\sum_{i=1}^{r_2}\abs{x_{\sigma^{-1}(l+i-1)}}}(A)
    \end{equation}
    for each summand in the relation. For each $1\leq i\leq r$, let $y_i\in A_{\abs{x_i}}$ be a representative of $x_i$ and, for each summand in the relation, let
    \[\rho^{(2)}\in A_{\abs{\mu^{(2)}}+\sum_{i=1}^{r_2}\abs{x_{\sigma^{-1}(l+i-1)}}+1}\]
    be an element such that
    \begin{equation}\label{massey_choice}
        d(\rho^{(2)})=\mu^{(2)}(y_{\sigma^{-1}(l)},\dots, y_{\sigma^{-1}(l+r_2-1)}).
    \end{equation}
    Such an element must exist by \eqref{vanishing}.
    We define the \emph{Massey product}
    \[\langle x_1,\dots,x_r\rangle_\Gamma\]
    as the element of
    \begin{equation}\label{quotient}
        \frac{H_{\abs{\Gamma}+\sum_{i=1}^r\abs{x_i}+1}(A)}{\sum\mu^{(1)}\left(x_{\sigma^{-1}(1)},\dots,x_{\sigma^{-1}(l-1)},H_{\abs{\mu^{(2)}}+\sum_{i=1}^{r_2}\abs{x_{\sigma^{-1}(l+i-1)}}+1}(A),x_{\sigma^{-1}(l+r_2)},\dots,x_{\sigma^{-1}(r_1)}\right)}
    \end{equation}
    represented by the homology class of
    \begin{equation}\label{representative}
        \sum(-1)^{\gamma}\mu^{(1)}(y_{\sigma^{-1}(1)},\dots,y_{\sigma^{-1}(l-1)},\rho^{(2)},y_{\sigma^{-1}(l+r_2)},\dots,y_{\sigma^{-1}(r)}),
    \end{equation}
    where
    \[\gamma  =\alpha+\abs{\mu^{(1)}}+(\abs{\mu^{(2)}}-1)\sum_{m=1}^{l-1}\abs{x_{\sigma^{-1}(m)}},\qquad
        \alpha  =\sum_{\substack{s<t\\\sigma(s)>\sigma(t)}}\abs{x_s}\abs{x_t}.\]
\end{definition}

\begin{remark}\label{basic_examples}
    The presentation of the associative operad $\O{A}$ is given by the $\mathbb{S}$-module $E$ with $E(2)_0=\kk[\mathbb{S}_2]$ generated as an $\mathbb{S}_2$-module by $\mu$, which represents the associative product, and $E(r)=0$ elsewhere. Moreover, $R$ is generated as an $\mathbb{S}$-module by
    \[\Gamma=\mu\circ_1\mu-\mu\circ_2\mu\in R(3),\]
    the associativity relation. The Massey products defined by $\Gamma$ are the same as the classical Massey products recalled in the introduction.

    Similarly, if $\O{L}$ is the Lie operad then $E(2)_0=k$ generated by $l$ with the sign action of $\mathbb{S}_2$, $E(r)=0$ otherwise, and $R$ is generated by
    \[\Gamma=(l\circ_1l)\cdot[()+(1\,2\,3)+(3\,2\,1)]\in R(3),\]
    the Jacobi identity. Massey products defined by $\Gamma$ coincide up to sing with the Lie-Massey products introduced in \cite{allday_1973_rational_whitehead_products,retah_1977_massey_operations_lie}.
\end{remark}

\begin{proposition}
    In Definition \ref{massey_product}, the chain \eqref{representative} is a cycle. Moreover, its homology class represents a well-defined element of the quotient \eqref{quotient}. Furthermore, different choices yield cycles \eqref{representative} whose cohomology classes run over all possible representatives of $\langle x_1,\dots,x_r\rangle_\Gamma$ in the quotient \eqref{quotient}.
\end{proposition}

\begin{proof}
    If we apply the differential $d$ of $A$ to \eqref{representative} we obtain the equation corresponding to the relation \eqref{relation} applied to $y_1,\dots,y_r\in A$, hence it vanishes.

    Once we choose $\rho^{(2)}$ as in \eqref{massey_choice}, the other possible choices are $\rho^{(2)}+\zeta^{(2)}$ where $\zeta^{(2)}\in A$ is a cycle of degree $\abs{\mu^{(2)}}+\sum_{i=1}^{r_2}\abs{x_{\sigma^{-1}(l+i-1)}}+1$. The difference between \eqref{representative} and the representative obtained from these other choices is
    \[\sum(-1)^{\gamma}\mu^{(1)}(y_{\sigma^{-1}(1)},\dots,y_{\sigma^{-1}(l-1)},\zeta^{(2)},y_{\sigma^{-1}(l+r_2)},\dots,y_{\sigma^{-1}(r)}).\]
    This is a cycle representing an element of the denominator of \eqref{quotient}. Conversely, for any element in the denominator of \eqref{quotient} we can take cycles $\zeta^{(2)}$ representing the elements in
    \[H_{\abs{\mu^{(2)}}+\sum_{i=1}^{r_2}\abs{x_{\sigma^{-1}(l+i-1)}}+1}(A).\]
    We can use these cycles to modify the initially chosen $\rho^{(2)}$ in order to obtain as \eqref{representative} all possible representatives of the Massey product in \eqref{quotient}.

    We may also choose another representative of some $x_i$, which must be of the form $y_i+d(w_i)$ for some $w_i\in A_{\abs{x_i}+1}$. This forces us to modify those $\rho^{(2)}$ where $i=\sigma^{-1}(l+j)$ for some $0\leq j<r_2$, but the situation is not too dramatic because we can take
    \[\rho^{(2)}+(-1)^{\gamma'}\mu^{(2)}(y_{\sigma^{-1}(l)},\dots, w_i,\dots,y_{\sigma^{-1}(l+r_2-1)}).\]
    Here $w_i$ occupies the slot $j+1$ and $\gamma'$ is the sum of the degrees of the symbols preceding $w_i$. The difference between the cycle produced by these new choices and \eqref{representative} is
    \begin{equation}\label{horror}
        \begin{split}
            \sum(-1)^{\gamma+\gamma'}\mu^{(1)}(y_{\sigma^{-1}(1)},\dots,\mu^{(2)}(y_{\sigma^{-1}(l)},\dots, w_i,\dots,y_{\sigma^{-1}(l+r_2-1)}),\dots,y_{\sigma^{-1}(r)})\\
            +\sum(-1)^{\gamma}\mu^{(1)}(y_{\sigma^{-1}(1)},\dots,d(w_i),\dots,\rho^{(2)},\dots,y_{\sigma^{-1}(r)}).
        \end{split}
    \end{equation}
    Here, the first summation is indexed by those summands of \eqref{relation} such that $i=\sigma^{-1}(l+j)$ for some $0\leq j<r_2$, and the second one is indexed by the rest of summands of \eqref{relation}. Moreover, in the latter, $d(w_i)$ occupies the place of $y_i$ in \eqref{representative}, which could also be after $\rho^{(2)}$, but we do not want to further overload the notation. We consider the chain
    \begin{equation}\label{horror2}
        \sum(-1)^{\gamma+\gamma''}\mu^{(1)}(y_{\sigma^{-1}(1)},\dots,w_i,\dots,\rho^{(2)},\dots,y_{\sigma^{-1}(r)})
    \end{equation}
    obtained, up to signs, by replacing $d(w_i)$ with $w_i$ in the second summation of \eqref{horror}. The element $\gamma''$ is the sum of the degrees of the symbols preceding $w_i$ here. A straightforward computation shows that the difference between \eqref{horror} and the differential of \eqref{horror2} is, up to sign,
    the relation \eqref{relation} applied to the elements $y_1,\dots, y_{i-1},w_i,y_{i+1},\dots,y_r\in A$, hence it is zero, i.e.~\eqref{horror} is the differential of \eqref{horror2}. This proves that the new choices yield a cycle in the same homology class as \eqref{representative} in $H_*(A)$.
\end{proof}

We can therefore regard the Massey product as a coset
\[\langle x_1,\dots,x_r\rangle_\Gamma\subset H_{\abs{\Gamma}+\sum_{i=1}^r\abs{x_i}+1}(A).\]
The denominator of \eqref{quotient} is often referred to as the \emph{indeterminacy} of this Massey product.

It is straightforward to check that Massey products are preserved by DG-$\O{O}$-algebra morphisms. Moreover, a weighted morphism of quadratic Koszul operads $f\colon\O{O}\to\O{P}$ induces a morphism between their $\mathbb{S}$-modules of relations $\bar{f}\colon R\to R'$. Given $\Gamma\in R$ and a DG-$\O{P}$-algebra $A$, a Massey product with respect to $\bar{f}(\Gamma)$ coincides with the Massey product with respect to $\Gamma$ upon restriction of scalars along $f$.

%% file: sections/massey_minimal.tex
Any DG-$\O{O}$-algebra $A$ has a \emph{minimal model}, consisting of an $\Oinf{\O{O}}$-algebra structure on $H_*(A)$ extending the induced $\O{O}$-algebra structure and an $\infty$-quasi-iso\-morph\-ism $f\colon H_*(A)\leadsto A$ (this notion is recalled below) whose underlying chain map sends a homology class to a representing cycle.

We now recall how an $\Oinf{\O{O}}$-algebra structure can be described in terms of the Koszul dual coaugmented cooperad $\K{\O{O}}$. The \emph{composite} of two $\mathbb{S}$-modules is
\begin{equation}\label{composite}
    M\circ N=\bigoplus_{r\geq 0}M(r)\otimes_{\mathbb{S}_r}N^{\otimes^r}.
\end{equation}
Here, a tensor $\mu\otimes (\nu_1\otimes\cdots\otimes \nu_r)\cdot\sigma$ is simply denoted by $(\mu;\nu_1,\dots, \nu_n)\cdot\sigma$. This defines a non-symmetric monoidal structure on $\mathbb{S}$-modules. The monoidal unit $\unit$, which is also the initial operad, is $\unit(1)=\kk$ concentrated in degree $0$ and $\unit(r)=0$ for $r\neq 1$. The \emph{infinitesimal composite} $M\circ_{(1)}N$ of two $\mathbb{S}$-modules is the sub-$\mathbb{S}$-module of $M\circ(\unit\oplus N)$ given by
\[M\circ_{(1)}N=\bigoplus_{r\geq 1}M(r)\otimes_{\mathbb{S}_r}\left(\bigoplus_{i=1}^r\unit^{\otimes^{i-1}}\otimes N\otimes \unit^{\otimes^{r-i}}\right).\]
A tensor $(x;1,\stackrel{i-1}{\dots},1,y,1,\stackrel{r-i}{\dots},1)\cdot\sigma$ here is denoted by $(x\circ_iy)\cdot\sigma$.
We denote the \emph{infinitesimal decomposition} \cite[\S6.1.4]{loday_vallette_2012_algebraic_operads} of $\K{\O{O}}$ by
\begin{equation}\label{sweedler_infinitesimal}
    \Delta_{(1)}\colon\K{\O{O}}\To \K{\O{O}}\circ_{(1)}\K{\O{O}},\qquad
    \Delta_{(1)}(\mu)=\sum_{(\mu)}(\mu^{(1)}\circ_{l}\mu^{(2)})\cdot\sigma.
\end{equation}
Here we use a Sweedler notation, like in \cite[\S10.1.2]{loday_vallette_2012_algebraic_operads}. 

An $\Oinf{\O{O}}$-algebra structure on a complex $A$ is given by morphisms
\begin{equation}\label{structure_maps}
    \begin{split}
        \K{\O{O}}(r)_{n}\otimes A_{p_1}\otimes\cdots\otimes A_{p_r} & \To A_{n+\sum_{i=1}^rp_i-1},  \\
        \mu\otimes x_1\otimes\cdots\otimes x_r                                  & \;\mapsto\;\mu(x_1,\dots,x_r),
    \end{split}
\end{equation}
satisfying
\[(\mu\cdot\sigma)(x_1,\dots,x_r)=(-1)^\alpha\mu(x_{\sigma^{-1}(1)},\dots,x_{\sigma^{-1}(r)}),\]
for any permutation $\sigma\in\mathbb{S}_r$, where $\alpha$ is as in Definition \ref{massey_product},
\begin{multline*}
    d(\mu(x_1,\dots,x_r))+\sum_{s=1}^r(-1)^{\beta}\mu(x_1,\dots,d(x_s),\dots,x_r)
    \\
    +\sum_{(\mu)}(-1)^{\gamma}\mu^{(1)}(x_{\sigma^{-1}(1)},\dots,\mu^{(2)}(x_{\sigma^{-1}(l)},\dots),\dots)=0,
\end{multline*}
where
\begin{align*}
    \beta  & =\abs{\mu}+\sum_{t=1}^{s-1}\abs{x_t},
\end{align*}
and $\gamma$ is again as in Definition \ref{massey_product}, 
and
\begin{equation}\label{one_vanishes}
    1(x)=0,
\end{equation}
where $1\in\K{\O{O}}(1)_0$ is given by the coaugmentation $\unit\to\K{\O{O}}$. This description can also be found in \cite[\S10.1.2]{loday_vallette_2012_algebraic_operads}.

Denote by 
\[s\colon \chain\To\chain\] 
the usual suspension functor in the category of chain complexes. The Koszul dual cooperad $\K{\O{O}}$ is cogenerated by $sE$ with corelations $s^2R$. In low weights, $(\K{\O{O}})^{(0)}=\unit$, $(\K{\O{O}})^{(1)}=sE$, $(\K{\O{O}})^{(2)}=s^2R$.

There is a canonical \emph{twisting morphism} $\kappa\colon\K{\O{O}}\to\O{O}$ with 
\begin{equation}\label{kappa}
     \kappa(s\mu)=\mu , \qquad  \mu\in E .
\end{equation}
It has homological degree $-1$ and vanishes in weight $\neq 1$. We can use it to pull back any $\O{O}$-algebra structure to an $\Oinf{\O{O}}$-algebra structure. We say that an $\Oinf{\O{O}}$-algebra structure on $A$ \emph{extends} a given graded $\O{O}$-algebra structure if
\begin{equation}\label{extends}
    (s\mu)(x_1,\dots,x_r)=\mu(x_1,\dots,x_r),\qquad \mu\in E(r).
\end{equation}

Given $\Gamma\in R$ as in \eqref{relation}, the infinitesimal decomposition satisfies
\begin{equation}\label{link}
        \Delta_{(1)}(s^2\Gamma)=1\circ_1(s^2\Gamma)+\sum_{i=1}^r(s^2\Gamma)\circ_i1
    +\sum(-1)^{\abs{\mu^{(1)}}}(s\mu^{(1)}\circ_l s\mu^{(2)})\cdot\sigma.
\end{equation}
This links the different meanings of $\circ_l$ in \eqref{relation} and \eqref{sweedler_infinitesimal}.

Let us now recall the definition of $\infty$-morphisms in terms of $\K{\O{O}}$. Denote the \emph{decomposition law} of $\K{\O{O}}$ by using a Sweedler notation like in \cite[\S5.8.1]{loday_vallette_2012_algebraic_operads},
\begin{gather*}
    \Delta\colon\K{\O{O}}\To \K{\O{O}}\circ\K{\O{O}},\\
    \Delta(\mu)=\sum_{[\mu]}(\nu;\nu^{1},\dots,\nu^{l})\cdot\tau.
\end{gather*}
An \emph{$\infty$-morphism} between $\Oinf{\O{O}}$-algebras $f\colon A\leadsto B$ is given by module morphisms
\begin{align*}
    \K{\O{O}}(r)_{n}\otimes A_{p_1}\otimes \cdots\otimes A_{p_r} & \To B_{n+\sum_{i=1}^r p_i},       \\
    \mu\otimes x_1\otimes\cdots\otimes x_r                                   & \;\mapsto\;f(\mu)(x_1,\dots,x_r),
\end{align*}
satisfying the following equations:
\[f(\mu\cdot\sigma)(x_1,\dots,x_r)=(-1)^\alpha f(\mu)(x_{\sigma^{-1}(1)},\dots,x_{\sigma^{-1}(r)}),\]
for any permutation $\sigma\in\mathbb{S}_r$, and
\begin{multline}\label{equation_infinity_morphism}
    \sum_{(\mu)}(-1)^{\gamma}f(\mu^{(1)})(x_{\sigma^{-1}(1)},\dots,\mu^{(2)}(x_{\sigma^{-1}(l)},\dots),\dots)\\
    -\sum_{[\mu]}(-1)^{\lambda}\nu(f(\nu^1)(x_{\tau^{-1}(1)},\dots),\dots,f(\nu^l)(\dots, x_{\tau^{-1}(r)}))\\
    =
    d(f(\mu)(x_1,\dots,x_r))-\sum_{t=1}^t(-1)^{\beta}f(\mu)(x_1,\dots,d(x_t),\dots,x_r),
\end{multline}
where $\lambda$ consists of adding up
\[\sum_{\substack{s<t\\\tau(s)>\tau(t)}}\abs{x_s}\abs{x_t}\]
and
\[|\nu^u| |x_{\tau^{-1}(v)}|\]
whenever $\nu^u$ appears after $x_{\tau^{-1}(v)}$.

The \emph{underlying morphism} of an $\infty$-morphism $f\colon A\leadsto B$ is the chain map $f(1)\colon A\to B$. An \emph{$\infty$-quasi-isomorphism} is an $\infty$-morphism whose underlying chain map is a quasi-isomorphism.

The \emph{composition} of $f$ with another $\infty$-morphism $g\colon B\leadsto C$ is given by
\begin{align*}
    (gf)(\mu)(x_1,\dots,x_r) & =\sum_{[\mu]}(-1)^\lambda g(\nu)(f(\nu_1)(x_{\tau^{-1}(1)},\dots),\dots,f(\nu_l)(\dots,x_{\tau^{-1}(r)})).
\end{align*}
In this way, we can consider the category of $\Oinf{\O{O}}$-algebras and $\infty$-morphisms between them. 

\begin{theorem}\label{massey_infinity}
    Given a DG-$\O{O}$-algebra $A$, an $\Oinf{\O{O}}$-algebra structure on $H_*(A)$ defining a minimal model $f\colon H_*(A)\leadsto A$, a relation $\Gamma\in R(r)$ as in \eqref{relation}, and $x_1,\dots,x_r\in H_*(A)$ satisfying the vanishing conditions in \eqref{vanishing}, then
    \[(s^2\Gamma)(x_1,\dots,x_r)\in \langle x_1,\dots,x_r\rangle_\Gamma,\]
    where $\langle x_1,\dots,x_r\rangle_\Gamma$ is a coset in 
    $H_{\abs{\Gamma}+\sum_{i=1}^r\abs{x_i}+1}(A)$.
\end{theorem}

\begin{proof}
    In order to construct a representative of the Massey product, we define
    \begin{align*}
        y_i&=f(1)(x_i),\\
        \rho^{(2)}&=-(-1)^{\lambda}f(s\mu^{(2)})(x_{\sigma^{-1}(l)},\dots,x_{\sigma^{-1}(l+r_2-1)}).
    \end{align*}
    This makes sense because $\mu^{(2)}\in E$ so
    \begin{align*}
        \Delta_{(1)}(\mu^{(2)})&=1\circ_1\mu^{(2)}+\sum_{i=1}^{r_2}\mu^{(2)}\circ_i1,&
        \Delta(\mu^{(2)})&=(1;\mu^{(2)})+(\mu^{(2)};1,\stackrel{r}{\dots},1),
    \end{align*}
    hence we derive \eqref{massey_choice} from \eqref{one_vanishes} and \eqref{equation_infinity_morphism}.

    The non-trivial part of the decomposition of $s^2\Gamma\in s^2R=\K{\O{O}}_{(2)}(r)$ is essentially the same as the non-trivial part of its infinitesimal decomposition \eqref{link},
    \begin{equation}\label{decomposition_relation}
        \begin{split}
            \Delta(s^2\Gamma)={}&(1;s^2\Gamma)+(s^2\Gamma;1,\stackrel{r}{\dots},1)\\
            &+\sum(-1)^{\abs{\mu^{(1)}}}(s\mu^{(1)};1,\stackrel{l-1}{\dots},1,s\mu^{(2)},1,\stackrel{r_1-l}{\dots},1)\cdot\sigma.
        \end{split}
    \end{equation}
    Using \eqref{equation_infinity_morphism},
    \begin{multline*}
        d(f(s^2\Gamma)(x_1,\dots,x_r))\\
        =f(1)((s^2\Gamma)(x_1,\dots,x_r))+\sum_{i=1}^rf((s^2\Gamma))(x_1,\dots,\underbrace{1(x_i)}_{=0},\dots,x_r)\\
        +\sum(-1)^{\abs{\mu^{(1)}}+\gamma}f(s\mu^{(1)})(x_{\sigma^{-1}(1)},\dots,(s\mu^{(2)})(x_{\sigma^{-1}(l)},\dots),\dots)\cdot\sigma\\
        -\underbrace{1((s^2\Gamma)(x_1,\dots,x_r))}_{=0}-\underbrace{\kappa(s^2\Gamma)}_{=0}(f(1)(x_1),\dots,f(1)(x_r))\\
        -\sum(-1)^{\gamma}(s\mu^{(1)})(f(1)(x_{\sigma^{-1}(1)}),\dots,f(s\mu^{(2)})(x_{\sigma^{-1}(l)},\dots),\dots)\\
        =f(1)((s^2\Gamma)(x_1,\dots,x_r))\\
        +\sum(-1)^{\abs{\mu^{(1)}}+\gamma}f(s\mu^{(1)})(x_{\sigma^{-1}(1)},\dots,\underbrace{\mu^{(2)}(x_{\sigma^{-1}(l),}\dots)}_{=0},\dots)\cdot\sigma\\
        -\sum(-1)^{\gamma}\mu^{(1)}(y_1,\dots,y_{l-1},\rho^{(2)},y_{l+1},\dots,y_{r_1})\\
        =f(1)((s^2\Gamma)(x_1,\dots,x_r))\\
        -\sum(-1)^{\gamma}\mu^{(1)}(y_1,\dots,y_{l-1},\rho^{(2)},y_{l+1},\dots,y_{r_1}).
    \end{multline*}
    Since $f(1)$ sends any homology class to a representing cycle, $f(1)((s^2\Gamma)(x_1,\dots,x_r))$ represents $(s^2\Gamma)(x_1,\dots,x_r)$. The last summation is, by definition, a cycle whose cohomology class represents $\langle x_1,\dots,x_r\rangle_\Gamma$, hence the proposition follows.
\end{proof}

As mentioned in the introduction, the structure morphisms of a minimal $\O{O}_\infty$-model $H_*(A)$ sometimes receive the name of Massey products, see e.g.~\cite{galvez-carrillo_tonks_vallette_2012_homotopy_batalinvilkovisky_algebras,loday_vallette_2012_algebraic_operads}. The closest operadic analogue of classical Massey products is Definition \ref{massey_product}. The previous result establishes the connection between both notions. We think that this is the strongest possible connection between minimal models and Massey-product-like operations since, even in the associative case, classical higher Massey products are in general unrelated to the higher operations on the minimal model, see \cite{buijs_moreno-fernandez_murillo_2020_infty_structures_massey}.

Recall that a DG-$\O{O}$-algebra $A$ is \emph{formal} if it is quasi-isomorphic to $H_*(A)$, or equivalently if there exists a minimal model $H_*(A)\leadsto A$ where $H_*(A)$ here is equipped with the $\Oinf{\O{O}}$-algebra structure pulled back from the induced $\O{O}$-algebra structure, in particular $s^2\Gamma$ operates trivially on $H_*(A)$ for any $\Gamma\in R$.

\begin{corollary}
    If $A$ is a formal DG-$\O{O}$-algebra then all Massey products in $H_*(A)$ vanish.
\end{corollary}

The converse is clearly not true, even for the associative operad. Indeed, any differential graded associative algebra with homology concentrated in even degrees has trivial Massey products for degree reasons, but not all of them are formal.

%% file: sections/universal.tex
Dimitrova defined in \cite{dimitrova_2012_obstruction_theory_operadic} a universal class for any algebra over a graded operad using minimal models, regardless of the ground field. This class is another obstruction to formality. She does not need the operad to be Koszul. Nevertheless, her general definition simplifies in the Koszul case in characteristic zero \cite[Proposition 4.11 and Remark 4.12]{dimitrova_2012_obstruction_theory_operadic} as we now recall.

The \emph{operadic cochain complex} $C_{\O{O}}^{\star,\ast}(A,M)$ of a graded $\O{O}$-algebra $A$ with coefficients in an $A$-module $M$  \cite[\S12.4]{loday_vallette_2012_algebraic_operads} is a bigraded complex given by
\[C_{\O{O}}^{w,t}(A,M)=\hom_{-t-w}((\K{\O{O}})^{(w)}\circ A,M)\]
with differential
\[d\colon C_{\O{O}}^{w,t}(A,M)\To C_{\O{O}}^{w+1,t}(A,M)\]
defined as
\begin{equation*}
    \begin{split}
        d(f)(\mu;x_1,\dots,x_r)={}&\sum_{(\mu)}(-1)^{\gamma'}\kappa(\mu^{(1)})(x_{\sigma^{-1}(1)},\dots,f(\mu^{(2)};x_{\sigma^{-1}(l)},\dots),\dots)\\
        &-\sum_{(\mu)}(-1)^{\gamma+\abs{f}}f(\mu^{(1)};x_{\sigma^{-1}(1)},\dots,\kappa(\mu^{(2)})(x_{\sigma^{-1}(l)},\dots),\dots).
    \end{split}
\end{equation*}
Here we use the notation in \eqref{sweedler_infinitesimal}, $\abs{f}=w+t$ if $f\in C_{\O{O}}^{w,t}(A,M)$,
\begin{align*}
    \gamma' & =\alpha+\abs{\mu^{(1)}}\abs{f}+(\abs{\mu^{(2)}}+\abs{f})\sum_{m=1}^{l-1}\abs{x_{\sigma^{-1}(m)}},
\end{align*}
and $\alpha$ and $\gamma$ are as in Definition \ref{massey_product}.
The \emph{operadic cohomology} is the cohomology of this complex,
$H_{\O{O}}^{\star,\ast}(A,M)$.

Assume now that $A$ is a differential graded $\O{O}$-algebra. Given a minimal model $f\colon H_*(A)\leadsto A$, Dimitrova's class \[\{\m\}\in H_{\O{O}}^{2,-1}(H_*(A),H_*(A))\] can be represented by the cocycle
$\m\in C_{\O{O}}^{2,-1}(H_*(A),H_*(A))$
defined by
\[\m(s^2\Gamma;x_1,\dots,x_r)=(s^2\Gamma)(x_1,\dots,x_r),\qquad \Gamma\in R(r), \quad x_i\in H_*(A).\]
Here we use that $(\K{\O{O}})^{(2)}=s^2R$. In the light of Theorem \ref{massey_infinity}, Dimitrova's class deserves to be called \emph{universal Massey product} after the universal Toda brackets introduced in \cite{baues_dreckmann_1989_cohomology_homotopy_categories}.

In the associative case, universal Massey products go back to \cite{kadeishvili_1988_structure_infty_algebra}. They were studied in detail in \cite{benson_krause_schwede_2004_realizability_modules_tate} with applications to the cohomology of finite groups. There, the authors compute an example where all Massey products vanish but the universal Massey product is non-trivial \cite[Example 5.15]{benson_krause_schwede_2004_realizability_modules_tate}.
Associative universal Massey products have also been recently applied to the existence and uniqueness of enhancements for  triangulated categories \cite{muro_2020_enhanced_finite_triangulated}.

All representatives of Dimitrova's class compute Massey products.

\begin{proposition}
    Let $\phi \in C_{\O{O}}^{2,-1}(H_*(A),H_*(A))$ be a cycle representing Dimitrova's class. Given a relation $\Gamma\in R(r)$ as in \eqref{relation} and $x_1,\dots,x_r\in H_*(A)$ satisfying the vanishing conditions in \eqref{vanishing} then
    \[\phi(s^2\Gamma;x_1,\dots,x_r)\in \langle x_1,\dots,x_r\rangle_{\Gamma}\]
    where $\langle x_1,\dots,x_r\rangle_{\Gamma}$ is a coset in $H_{\abs{\Gamma}+\sum_{i=1}^r\abs{x_i}+1}(A)$.
\end{proposition}

\begin{proof}
    We take a cochain $\xi \in C_{\O{O}}^{1,-1}(H_*(A),H_*(A))$ relating $\phi$ to the representative $\m$ defined from a minimal model, $d(\xi)=\phi-\m$. 
    Since
    \[
        \phi(s^2\Gamma;x_1,\dots,x_r)=\m(s^2\Gamma;x_1,\dots,x_r)+d(\xi)(s^2\Gamma;x_1,\dots,x_r),
    \]
    by Theorem \ref{massey_infinity} it suffices to check that the last summand lies in the denominator of \eqref{quotient}. Using \eqref{link}, the definition of the canonical twisting morphism $\kappa$ in \eqref{kappa}, and the vanishing condition \eqref{vanishing},
    \begin{multline*}
        d(\xi)(s^2\Gamma;x_1,\dots,x_r)=\\
        \sum(-1)^{\gamma}\mu^{(1)}(x_{\sigma^{-1}(1)},\dots,\xi(s\mu^{(2)};x_{\sigma^{-1}(l)},\dots),\dots)\\
        +\sum(-1)^{\alpha+\abs{\mu^{(2)}}\sum_{m=1}^{l-1}\abs{x_{\sigma^{-1}(m)}}}\xi(s\mu^{(1)};x_{\sigma^{-1}(1)},\dots,\underbrace{\mu^{(2)}(x_{\sigma^{-1}(l)},\dots)}_{=0},\dots).
    \end{multline*}
    This concludes the proof since $|s\mu^{(2)}|=\abs{\mu^{(2)}}+1$ so
    \[\xi(s\mu^{(2)};x_{\sigma^{-1}(l)},\dots)\in H_{\abs{\mu^{(2)}}+\sum_{i=1}^{r_2}\abs{x_{\sigma^{-1}(l+i-1)}}+1}(A).\]
\end{proof}

\begin{remark}\label{hypotheses}
    Despite we nominally work with a graded Koszul operad $\O{O}$ defined over a field of characteristic zero, these hypotheses can be weakened. In \S\ref{firstsection} we only need an operad $\O{O}$ with some quadratic relation $\Gamma$, we do not need all relations to be quadratic or the operad to be Koszul, nor we need the ground ring to be a field. For Theorem \ref{massey_infinity} we need a quadratic operad and minimal models, but the operad need not be Koszul. Minimal models can be obtained over arbitrary ground fields from the usual transfer theorem \cite[\S10.3]{loday_vallette_2012_algebraic_operads}. Actually, this theorem allows for the construction of a minimal model for an $\O{O}$-algebra $A$ with underlying cofibrant complex and projective homology $H_*(A)$ over an arbitrary ground ring. We consider the projective model structure on chain complexes. Indeed, by the projectivity hypothesis we can choose a representing cycle selection chain map $i\colon H_*(A)\to A$, which is a quasi-isomorphism by construction. Since $i$ is a quasi-isomorphism between fibrant-cofibrant objects, it  admits a homotopy retraction $p\colon A\to H_*(A)$. The composite $ri$ must be the identity since homotopic self maps of $H_*(A)$ must be equal because its differential is trivial. This is all we need for the transfer theorem. If we want to drop the cofibrancy hypothesis on $A$ we must be able to replace it with an $\O{O}$-algebra $\tilde A$ with underlying cofibrant complex. This can be achieved if $\O{O}$ is $\mathbb{S}$-cofibrant, see \cite[Theorem 12.3.A and Proposition 12.3.2]{fresse_2009_modules_operads_functors}. In this case, we first take such a replacement $\tilde A\to A$ and then construct a minimal model for $\tilde A$ by using the transfer theorem. This yields a minimal model for $A$ after composing with $\tilde A\to A$. For the previous description of Dimitrova's class we also need $\O{O}$ to be Koszul. This condition is also necessary for minimal models to be homotopically meaningful, but this is independent from the definition of Massey products.
\end{remark}

%% file: sections/application.tex
We have already mentioned in Remark \ref{basic_examples} that usual Massey products for associative algebras and Lie-Massey products fit into our framework. We finish this paper by computing non-trivial Massey products in DG-Gerstenhaber and hypercommutative algebras arising in nature.

\begin{definition}
    A \emph{Gerstenhaber algebra} $A$ is a graded vector space equipped with a commutative algebra structure with product $a b$ and a degree $1$ Lie algebra structure with bracket $[a,b]$ satisfying the following relation, known as \emph{Gerstenhaber relation}, 
    \begin{align*}
        [a,b\cdot c] & =[a,b]\cdot c+(-1)^{(\abs{a}-1)\abs{b}}b\cdot[a,c].
    \end{align*}
\end{definition}


The operad $\O{G}$ governing Gerstenhaber algebras is generated by the $\mathbb{S}$-module $E$ with
\begin{align*}
    E(2)_0   & = \kk\cdot c, &
    E(2)_{1} & = \kk\cdot l,
\end{align*}
with the trivial action of $\mathbb{S}_2$ and $E(r)_n=0$ elsewhere. The elements $c$ and $l$ correspond to the commutative product, $a\cdot b=c(a,b)$ and the Lie bracket $[a,b]=(-1)^{\abs{a}}l(a,b)$ of a $\O{G}$-algebra $A$. The sub-$\mathbb{S}$-module of relations $R\subset\F(E)^{(2)}$ is generated by the arity $3$ elements
\begin{gather*}
    c\circ_1c-c\circ_2c,\\
    (l\circ_1l)\cdot[()+(1\,2\,3)+(3\,2\,1)],\\
    l\circ_2 c-c\circ_1l-(c\circ_2l)\cdot(1\;2),
\end{gather*}
in degrees $0,2,1$. These elements reflect the associative relation, the Jacobi identity, and the Gerstenhaber relation. The operad $\O{G}$ is Koszul by \cite{Getzler1994}.

The following examples of $\O{G}$-algebras will be cochain complexes, rather than chain complexes like until now. Therefore, we must switch to cohomological degrees in the usual way $H^n=H_{-n}$.

Recall that the \emph{Chevalley--Eilenberg complex} $C^*(\L{g},\kk)$ of a finite-dimensional Lie algebra $\L{g}$ is the exterior algebra $\bigwedge\L{g}^*$, i.e.~the free graded unital commutative algebra generated by the dual vector space $\L{g}^*$ in degree $1$. This vector space $\L{g}^*$ is a Lie coalgebra whose structure map
\[C^1(\L{g},\kk)=\L{g}^*\stackrel{d}{\To}\wedge^2\L{g}^*=C^2(\L{g},\kk)\]
defines the differential of the DG-commutative algebra $C^*(\L{g},\kk)$. This complex computes the cohomology $H^*(\L{g},\kk)$ of $\L{g}$ with coefficients in the trivial $\L{g}$-module $\kk$. If $\kk=\mathbb{Q}$ and $\L{g}$ is nilpotent, then $C^*(\L{g}\otimes\mathbb{R},\mathbb{R})$ is quasi-isomorphic to the de Rham complex $\Omega(M)$ of a compact homogeneous space $M$ of the simply connected Lie group with Lie algebra $\L{g}\otimes\mathbb{R}$ \cite{Nomizu1954}.

If $\L{g}$ is a Lie bialgebra, i.e.~a Lie algebra equipped with a compatible Lie coalgebra structure, then $C^*(\L{g},\kk)$ becomes a Gerstenhaber algebra \cite[Example 3.1]{kosmann-schwarzbach_1995_exact_gerstenhaber_algebras}. Its Lie bracket is the only Gerstenhaber extension of the Lie bracket on $\L{g}^*$ dual to the Lie coalgebra structure of $\L{g}$. Over $\kk=\mathbb{R}$, Lie bialgebras correspond to simply connected Poisson Lie groups in the same way as Lie algebras correspond to simply connected Lie groups \cite{drinfelcprime_d_1983_hamiltonian_structures_lie}. In particular, if $\kk=\mathbb{Q}$ and $\L{g}$ is nilpotent as above, the DG-Gerstenhaber algebra $C^*(\L{g}\otimes\R,\R)$ is an invariant of $M$ as a Poisson manifold.

We consider the $3$-dimensional Heisenberg Lie algebra $\L{h}$. It is defined over any characteristic zero field. For $\kk=\R$, the associated homogeneous space $M$ is the well-known Heisenberg manifold. The dual $\L{h}^*$ has a basis $\{x,y,z\}$
with Lie cobracket $d\colon \L{h}^*\to\wedge^2\L{h}^*$ defined by
\begin{align*}
    d(x) & =0,   &
    d(y) & =0,   &
    d(z) & =x y.
\end{align*}
We endow $\L{h}$ with the Lie bialgebra structure such that the induced Lie bracket on $\L{h}^*$ is given by \begin{align*}
    [z,x] & =x,   &
    [z,y] & =x+y, &
    [x,y] & =0.
\end{align*}
The Lie bialgebra $\L{h}^*$ is denoted by $\L{r}_3$ in \cite[Theorem 5.2]{farinati_jancsa_2015_three_dimensional_real}.

The cohomology $H^*(\L{h},\kk)$ has the following bases,
\begin{align*}
    H^0(\L{h},\kk) & \supset\{1\},                           &
    H^1(\L{h},\kk) & \supset\{\bar x,\bar y\},                 \\
    H^2(\L{h},\kk) & \supset\{\overline{xz},\overline{yz}\}, &
    H^3(\L{h},\kk) & \supset\{\overline{xyz}\}.
\end{align*}
We overline representing cocycles.
The product is trivial, except for the product with the unit $1$ and
\[\bar x\overline{yz}=-\bar y\overline{xz}=\overline{xyz}.\]

\begin{proposition}\label{lie}
    Let $\Gamma=l\circ_2 c-c\circ_1l-(c\circ_2l)(1\;2)\in R$ be the Gerstenhaber relation. For the previous Lie bialgebra $\L{h}$, the following Massey product is well defined, it has no indeterminacy, and it is non-trivial,
    \[\langle\overline{yz},\bar x,\bar y\rangle_\Gamma=2\overline{xz}\in H^2(\L{h},\kk).\]
\end{proposition}

\begin{proof}
    The Massey product is defined since, in $C^*(\L{h},\kk)$,
    \begin{align*}
        x y     & = d(z),          \\
        [y z,x] & = y[z,x]+[y,x] z \\
                & =y x             \\
                & =-d(z),          \\
        [y z,y] & = y[z,y]+[y,y] z \\
                & =y (x+y)         \\
                & =-d(z).
    \end{align*}
    The indeterminacy is the following sum
    \[[\overline{y z},H^1(\L{h},\kk)]+ H^1(\L{h},\kk)\bar y+\bar{x} H^1(\L{h},\kk)\subset H^2(\L{h},\kk).\] The last two summands vanish since $x y=d(z)$ and $x^2= y^2=0$. The first one too since $[y z,x]=[y z,y]=-d(z)$. The Massey product is then represented by
    \begin{align*}
        -[y z,z] -(-z) y - x (-z) & =-y[z,z]-[y,z] z + z y + x z \\
                                  & =(x+y) z+ z y + x z          \\
                                  & =2 xz.
    \end{align*}
\end{proof}

The de Rham complex of a general Poisson manifold carries a canonical DG-Gerstenhaber algebra structure \cite{koszul_1985_crochet_schoutennijenhuis_cohomologie}. In the previous example, despite $C^*(\L{h},\kk)$ is quasi-isomorphic to $\Omega(M)$ as a DG-commutative algebra ($M$ the Heisenberg manifold) they are not quasi-isomorphic as DG-Gerstenhaber algebras. This can be detected in cohomology, since the latter induces the trivial Lie bracket, unlike the former, which satisfies
\[[\overline{yz},\overline{yz}]=-2\overline{xyz}.\]
We will revisit this example later because it is related to the operad we consider next.

\begin{definition}
    A \textit{hypercommutative algebra} $A$ is a graded vector space equipped with arity $r$ multiplications of degree $2(r-2)$ for any $r\geq 2$,
    \begin{align*}
        A\otimes\stackrel{r}{\cdots}\otimes A & \To A,                       \\
        x_1\otimes\cdots\otimes x_r           & \;\mapsto\; (x_1,\dots,x_r),
    \end{align*}
    satisfying
    \begin{align*}
        (x_1,\dots,x_r)                                                & =
        \pm(x_{\sigma(1)},\dots,x_{\sigma(r)}),\qquad \sigma\in\mathbb{S}_r, \\
        \sum_{S_1\amalg S_2=\{1,\dots,r\}}\pm((a,b,x_{S_1}),c,x_{S_2}) & =
        \sum_{S_1\amalg S_2=\{1,\dots,r\}}\pm(a,(b,c,x_{S_1}),x_{S_2}).
    \end{align*}
    Here, given an ordered subset $S=\{i_1<\cdots<i_s\}\subset\{1,\dots,r\}$, $x_S$ denotes the sequence $x_{i_1},\dots,x_{i_s}$. In both equations, signs are just determined by the Koszul sign rule with respect to the initial ordering $a,b,c,x_1,\dots,x_r$.
\end{definition}

The arity $2$ operation defines a commutative algebra structure on $A$,
\[a b=(a,b).\]
Hence, hypercommutative algebras are commutative algebras with extra structure.

The relation involving just the arity $2$ and $3$ operations is
\begin{equation}\label{hypercommutative_key_relation}
    (a b, c, x)+(-1)^{\abs{c}\abs{x}}(a,b,x) c=a (b,c,x)+(a,b c,x).
\end{equation}

The operad $\O{H}$ governing hypercommutative algebras is generated by the $\mathbb{S}$-module $E$ with
\[E(r)_{2(r-2)}=\kk\cdot m_r,\qquad r\geq2,\]
with the trivial action of $\mathbb{S}_r$ and $E(r)_n=0$ elsewhere. The element $m_r$ corresponds to the arity $r$ operation. The sub-$\mathbb{S}$-module of relations $R\subset\F(E)^{(2)}$ is generated by the degree $2r$ elements
\[\sum_{\substack{p+q=r\\(p,q)\text{-shuffles}}}(m_{p+2}\circ_1 m_{q+2})\cdot
    (3\;\cdots\; q+3)
    \sigma
    -
    \sum_{\substack{p+q=r\\(p,q)\text{-shuffles}}}(m_{p+2}\circ_2 m_{q+2})\cdot\sigma.\]
Here $r\geq 0$, we consider $(p,q)$-shuffles $\sigma\in\mathbb{S}_r$, and we regard $\mathbb{S}_r\subset \mathbb{S}_{3+r}$ as the subgroup permuting the last $r$ elements. This operad is Koszul \cite{Getzler1995Operads-and-moduli-spaces-of-genus-0-Riemann-surfaces}.

For $r=0$, the previous relation is the associativity relation for the binary operation $m_2$. For $r=1$ it is the relation
\begin{equation}\label{hypercommutative_relation}
    m_3\circ_1m_2+(m_2\circ_1m_3)\cdot (3\; 4)-m_3\circ_2m_2-m_2\circ_2m_3
\end{equation}
corresponding to \eqref{hypercommutative_key_relation}.

A common way of constructing DG-hypercommutative algebras is from DG-Batalin-Vilkovisky algebras equipped with a trivialization of the Batalin-Vilkovisky operator. We now describe this approach following \cite[\S1]{Khoroshkin2013}.

\begin{definition}
    A \emph{Batalin-Vilkovisky algebra} $A$ is a Gerstenhaber algebra equipped with a degree $1$ morphism, the \emph{Batalin-Vilkovisky operator},
    \[\Delta\colon A\To A,\]
    such that $\Delta^2=0$ and
    \[[a,b]=\Delta(a\cdot b)-\Delta(a)\cdot b-(-1)^{\abs{a}}a\cdot\Delta(b).\]
\end{definition}

If $M$ is a Poisson manifold, then $\Omega(M)$ is in fact a Batalin-Vilkovisky algebra with $\Delta=[d,i_w]$. Here $i_w\colon \Omega(M)\to\Omega(M)$ is the interior product with a Poisson bivector $w$, which is a degree $-2$ and order $\leq 2$ differential operator of the DG-commutative algebra $\Omega(M)$, see \cite{koszul_1985_crochet_schoutennijenhuis_cohomologie}.



If $A$ is a DG-Batalin-Vilkovisky algebra with differential $d$ and $t$ is a formal parameter of degree $-2$, $A\serie{t}$ carries two chain complex structures with differentials $d$ and $d+\Delta t$, respectively, with $t$ a cycle.

\begin{definition}
    A \emph{trivialization} of a DG-Batalin-Vilkovisky algebra is a chain isomorphism
    \[\Phi(z)\colon (A\serie{t},d+\Delta t)\cong (A\serie{t},d).\]
\end{definition}

The operad $\O{H}$ for hypercommutative algebras is the homology of the DG-operad $\O{BV}/\Delta$ encoding DG-Batalin-Vilkovisky algebras equipped with a trivialization. The quotient $\O{BV}/\Delta$ is just the notation used by \cite{Khoroshkin2013}. The latter operad is actually formal. An explicit quasi-isomorphism from the former to the latter is constructed in \cite{Khoroshkin2013}, providing a direct way of endowing a trivialized DG-Batalin-Vilkovisky algebra with a DG-hypercommutative algebra structure.

\begin{remark}\label{colleagues}
    A trivialization of $\Omega(M)$ for $M$ a Poisson manifold is
    \[e^{i_wt}\colon(\Omega(M)\serie{t},d+\Delta t)\cong (\Omega(M)\serie{t},d),\]
    see \cite[Theorem 3.7]{dotsenko_shadrin_vallette_2015_rham_cohomology_homotopy}. The arity $3$ operation of the induced DG-hypercommutative algebra structure on $\Omega(M)$ vanishes. The formula for this operation is spelled out in \cite[Example 1.2]{Khoroshkin2013} (second formula). Therein, the binary product is called $m$, the arity $3$ product is $\theta_3$, and $i_w=\phi_1$. The vanishing of this formula is equivalent to $i_w$ being a differential operator of order $\leq 2$. We do not know of any computation showing that, for a certain Poisson manifold $M$, the DG-hypercommutative algebra structure on $\Omega(M)$ does not reduce to the usual DG-commutative algebra structure with trivial products of arities $\geq 3$.
\end{remark}

We endow the Chevalley--Eilenberg complex $C^*(\kk^2\oplus\L{h},\kk)$ of the direct sum of a $2$-dimensional abelian Lie algebra and the Heisenberg Lie algebra with the trivial Batalin-Vilkovisky operator $\Delta=0$, and hence the trivial Gerstenhaber Lie bracket. Despite the Batalin-Vilkovisky operator is already trivial, we can take a non-trivial trivialization in order to endow $A$ with an interesting DG-hypercommutative algebra structure.

In $C^*(\kk^2\oplus\L{h},\kk)$, we consider the basis given by the products of the elements of the bases $\{v,w\}\subset(\kk^2)^*$ and $\{x,y,h\}\subset\L{h}^*$ in lexicographic order. The product of two elements of this basis is either another element in the basis, up to sign, or zero. We consider the degree $-2$  endomorphism $i\colon C^*(\kk^2\oplus\L{h},\kk)\to C^*(\kk^2\oplus\L{h},\kk)$ defined on the basis as
\[i(v w x)= y\]
and zero otherwise. It is tedious but straightforward to check that $i$ is a chain map. Hence,
\[e^{it}\colon (C^*(\kk^2\oplus\L{h},\kk)\serie{t},d)\cong (C^*(\kk^2\oplus\L{h},\kk)\serie{t},d)\]
is a trivialization by \cite[\S1.1]{Khoroshkin2013}. We endow $C^*(\kk^2\oplus\L{h},\kk)$ with the hypercommutative algebra structure induced by this trivialization. Below we use that $H^*(\kk^2\oplus\L{h},\kk)=(\bigwedge(\kk^2)^*)\otimes H^*(\L{h},\kk)$.

\begin{proposition}\label{hypercommutative}
    Let $\Gamma$ be the relation in \eqref{hypercommutative_relation}. Consider the DG-hypercommutative algebra structure on $C^*(\kk^2\oplus\L{h},\kk)$ we have just defined. The following Massey product is well defined and non-trivial, moreover
    \[\langle \bar v\bar w,\bar  v \bar x, \bar x, \bar x\rangle_\Gamma\cap (\kk^2\otimes H^2(\L{h},\R)) =\{\bar v \overline{x z}\}.\]
\end{proposition}

\begin{proof}
    Apart from the commutative binary product on $C^*(\kk^2\oplus\L{h},\kk)$, we only need to know the arity $3$ product, which has degree $-2$ since $C^*(\kk^2\oplus\L{h},\kk)$ is a cochain complex. This arity $3$ product can be computed from the binary product and the chain map $i$ by means of the second formula in \cite[Example 1.2]{Khoroshkin2013}, already mentioned in Remark \ref{colleagues}.

    The Massey product is defined since
    \begin{align*}
        v w v x & =- v^2 w x=0, &
        v x x   & =v x^2=0.
    \end{align*}
    Moreover,
    \[(v x,x,x)=0\]
    by hypercommutativity, since we have two repeated entries of odd degree. We would like to stress that the three previous vanishing equations hold at the level of cochains. This will simplify the computation of the Massey product.
    Furthermore, since $i$ vanishes in dimensions $\neq 3$,
    \begin{align*}
        (v w, v x,  x) & = - v x i(v w x) \\
                       & = - v x i(v w x) \\
                       & = - v x y        \\
                       & = d(v z).
    \end{align*}

    By the previous computations, the Massey product is represented by
    \[-vzx=vxz,\]
    which is a cycle whose cohomology class belongs to the direct summand  $\kk^2\otimes H^2(\L{h},\R)\subset H^3(\kk^2\oplus\L{h},\R)$.

    The indeterminacy is the following sum,
    \[(H^3(\kk^2\oplus\L{h},\kk),\bar x,\bar x)+H^2(\kk^2\oplus\L{h},\kk)\bar{x}
        +\bar v\bar wH^1(\kk^2\oplus\L{h},\kk)+(\bar v\bar w,H^2(\kk^2\oplus\L{h},\kk),\bar x).\]
    The first summand vanishes by hypercommutativity. The second and third summands decompose as
    \begin{align*}
        H^2(\kk^2\oplus\L{h},\kk)\cdot \bar{x}      & = H^2(\L{h},\kk)\cdot\bar x\oplus \kk^2\otimes (H^1(\L{h},\kk)\cdot\bar x)\oplus \wedge^2\kk\otimes\kk\cdot \bar x \\
                                                    & = \kk\cdot \{\bar x \overline{yz},\bar v\bar w\bar x\},                                                            \\
        \bar v\bar w\cdot H^1(\kk^2\oplus\L{h},\kk) & = \kk\cdot \{\bar v\bar w\bar x,\bar v\bar w\bar y\}.
    \end{align*}
    We now compute the fourth summand. A basis of $H^2(\kk^2\oplus\L{h},\kk)$ is
    \[\{\bar v\bar w, \bar v\bar x, \bar v\bar y, \bar w\bar x, \bar w\bar y, \overline{xz}, \overline{yz}\}.\]
    The possible arity $3$ products are
    \begin{align*}
        (vw,vw,x) & = 0,                     \\
        (vw,vx,x) & = d(vz),                 \\
        (vw,vy,x) & = -vyi(vwx)=-vy^2=0,     \\
        (vw,wx,x) & = -wxi(vwx)=-wxy=d(wz),  \\
        (vw,wy,x) & = -wyi(vwx)=-wy^2=0,     \\
        (vw,xz,x) & = -xzi(vwx)=-xzy=xyz,    \\
        (vw,yz,x) & = -yzi(vwx)=-yzy=zy^2=0.
    \end{align*}
    Here we use the same arguments as above. Exactly one of these arity $3$ products is non-trivial in cohomology. Hence the fourth summand has basis
    \[\{\bar{x}\overline{yz}\},\]
    and a basis of the whole indeterminacy is
    \[\{\bar v\bar w\bar x,\bar v\bar w\bar y,\bar{x}\overline{yz}\}.\]
    Therefore the intersection of the indeterminacy with  $\kk^2\otimes H^2(\L{h},\R)$ in $H^3(\kk^2\oplus\L{h},\kk)$ is trivial.
\end{proof}

The existence of a non-trivial Massey product associated to the relation \eqref{hypercommutative_relation} shows that the DG-hypercommutative algebra structure on $C^*(\kk^2\oplus\L{h},\kk)$ is not quasi-isomorphic to that of $\Omega(S^1\times S^1\times M)$ induced by a Poisson structure on $S^1\times S^1\times M$ ($M$ the Heisenberg manifold), since the latter would have trivial arity $3$ product, see Remark \ref{colleagues}. This also follows from the fact that the hypercommutative algebra $H^*(\kk^2\oplus\L{h},\kk)$ has a non-trivial arity $3$ product, namely,
\[(\bar v\bar w,\overline{xz},\bar x)=\bar x\overline{yz},\]
computed in the previous proof.

%% file: massey.bbl
\begin{thebibliography}{BMFM20}

  \bibitem[All73]{allday_1973_rational_whitehead_products}
  Christopher Allday, \emph{Rational {{Whitehead}} products and a spectral
    sequence of {{Quillen}}}, Pacific Journal of Mathematics \textbf{46} (1973),
  no.~2, 313--323.

  \bibitem[All77]{allday_1977_rational_whitehead_products}
  \bysame, \emph{Rational {{Whitehead}} products and a spectral sequence of
    {{Quillen}}. {{II}}}, Houston Journal of Mathematics \textbf{3} (1977),
  no.~3, 301--308. \MR{474288}

  \bibitem[BD89]{baues_dreckmann_1989_cohomology_homotopy_categories}
  Hans-Joachim Baues and Winfried Dreckmann, \emph{The cohomology of homotopy
    categories and the general linear group}, $K$-Theory \textbf{3} (1989),
  no.~4, 307--338. \MR{1047191}

  \bibitem[BKS04]{benson_krause_schwede_2004_realizability_modules_tate}
  David Benson, Henning Krause, and Stefan Schwede, \emph{Realizability of
    modules over {{Tate}} cohomology}, Trans. Amer. Math. Soc. \textbf{356}
  (2004), no.~9, 3621--3668. \MR{2055748}

  \bibitem[BMFM20]{buijs_moreno-fernandez_murillo_2020_infty_structures_massey}
  Urtzi Buijs, José~M. Moreno-Fernández, and Aniceto Murillo, \emph{$a_\infty$
    structures and {M}assey products}, Mediterranean Journal of Mathematics
  \textbf{17} (2020), no.~1, Paper No. 31, 15. \MR{4045178}

  \bibitem[CMW16]{Campos2016}
  Ricardo Campos, Sergei Merkulov, and Thomas Willwacher, \emph{The frobenius
    properad is koszul}, Duke Mathematical Journal \textbf{165} (2016),
  2921--2989.

  \bibitem[Dim12]{dimitrova_2012_obstruction_theory_operadic}
  Boryana Dimitrova, \emph{Obstruction theory for operadic algebras}, Ph.D.
  thesis, Universit\"ats- und Landesbibliothek Bonn, May 2012.

  \bibitem[{Dri}83]{drinfelcprime_d_1983_hamiltonian_structures_lie}
  V.~G. {Drinfel'd}, \emph{Hamiltonian structures on {{Lie}} groups, {{Lie}}
  bialgebras and the geometric meaning of classical {{Yang}}-{{Baxter}}
  equations}, Dokl. Akad. Nauk SSSR \textbf{268} (1983), no.~2, 285--287.
  \MR{688240}

  \bibitem[DSV15]{dotsenko_shadrin_vallette_2015_rham_cohomology_homotopy}
  Vladimir Dotsenko, Sergey Shadrin, and Bruno Vallette, \emph{De {{Rham}}
    cohomology and homotopy {{Frobenius}} manifolds}, J. Eur. Math. Soc.
  \textbf{17} (2015), no.~3, 535--547.

  \bibitem[FJ15]{farinati_jancsa_2015_three_dimensional_real}
  Marco~A. Farinati and A.~Patricia Jancsa, \emph{Three dimensional real {{Lie}}
    bialgebras}, Rev. Un. Mat. Argentina \textbf{56} (2015), no.~1, 27--62.
  \MR{3361840}

  \bibitem[Fre09]{fresse_2009_modules_operads_functors}
  Benoit Fresse, \emph{Modules over operads and functors}, Lecture {{Notes}} in
    {{Mathematics}}, vol. 1967, {Springer-Verlag, Berlin}, 2009. \MR{2494775}

  \bibitem[Get95]{Getzler1995Operads-and-moduli-spaces-of-genus-0-Riemann-surfaces}
  E.~Getzler, \emph{Operads and moduli spaces of genus {$0$} {R}iemann surfaces},
  The moduli space of curves ({T}exel {I}sland, 1994), Progr. Math., vol. 129,
  Birkh\"{a}user Boston, Boston, MA, 1995, pp.~199--230. \MR{1363058}

  \bibitem[GJ94]{Getzler1994}
  Ezra Getzler and J.~D.~S. Jones, \emph{Operads, homotopy algebra and iterated
    integrals for double loop spaces}, http://arxiv.org/abs/hep-th/9403055, 3
  1994.

  \bibitem[GK94]{ginzburg_kapranov_1994_koszul_duality_operads}
  Victor Ginzburg and Mikhail Kapranov, \emph{Koszul duality for operads}, Duke
  Math. J. \textbf{76} (1994), no.~1, 203--272. \MR{1301191}

  \bibitem[GTV12]{galvez-carrillo_tonks_vallette_2012_homotopy_batalinvilkovisky_algebras}
  Imma {G{\'a}lvez-Carrillo}, Andrew Tonks, and Bruno Vallette, \emph{Homotopy
    {{Batalin}}-{{Vilkovisky}} algebras}, J. Noncommut. Geom. \textbf{6} (2012),
  no.~3, 539--602. \MR{2956319}

  \bibitem[Kad88]{kadeishvili_1988_structure_infty_algebra}
  T.~V. Kadeishvili, \emph{The structure of the ${{A}}(\infty)$-algebra, and the
      {{Hochschild}} and {{Harrison}} cohomologies}, Trudy Tbiliss. Mat. Inst.
  Razmadze Akad. Nauk Gruzin. SSR \textbf{91} (1988), 19--27. \MR{1029003}

  \bibitem[KMS13]{Khoroshkin2013}
  A.~Khoroshkin, N.~Markarian, and S.~Shadrin, \emph{Hypercommutative operad as a
    homotopy quotient of bv}, Communications in Mathematical Physics \textbf{322}
  (2013), 697--729.

  \bibitem[Kos85]{koszul_1985_crochet_schoutennijenhuis_cohomologie}
  Jean-Louis Koszul, \emph{Crochet de {{Schouten}}-{{Nijenhuis}} et cohomologie},
  Ast\'erisque (1985), no.~Num\'ero Hors S\'erie, 257--271. \MR{837203}

  \bibitem[{Kos}95]{kosmann-schwarzbach_1995_exact_gerstenhaber_algebras}
  Y.~{Kosmann-Schwarzbach}, \emph{Exact {{Gerstenhaber}} algebras and {{Lie}}
    bialgebroids}, Acta {{Applicandae Mathematicae}} \textbf{41} (1995),
  153--165. \MR{1362125}

  \bibitem[Liv15]{livernet_2015_nonformality_swisscheese_operad}
  Muriel Livernet, \emph{Non-formality of the {{Swiss}}-cheese operad}, J. Topol.
  \textbf{8} (2015), no.~4, 1156--1166. \MR{3431672}

  \bibitem[LPWZ09]{lu_palmieri_wu_zhang_2009_ainfinity_structure_extalgebras}
  D.-M. Lu, J.H. Palmieri, Q.-S. Wu, and J.J. Zhang, \emph{A-infinity structure
    on {{Ext-algebras}}}, Journal of Pure and Applied Algebra \textbf{213}
  (2009), no.~11, 2017--2037.

  \bibitem[LV12]{loday_vallette_2012_algebraic_operads}
  Jean-Louis Loday and Bruno Vallette, \emph{Algebraic operads}, Grundlehren Der
    {{Mathematischen Wissenschaften}} [{{Fundamental Principles}} of
        {{Mathematical Sciences}}], vol. 346, {Springer, Heidelberg}, 2012.
  \MR{2954392}

  \bibitem[Mas69]{massey_1969_higher_order_linking}
  W.~S. Massey, \emph{Higher order linking numbers}, Conf. on {{Algebraic
          Topology}} ({{Univ}}. of {{Illinois}} at {{Chicago Circle}}, {{Chicago}},
  {{Ill}}., 1968), {Univ. of Illinois at Chicago Circle, Chicago, Ill.}, 1969,
  pp.~174--205. \MR{0254832}

  \bibitem[MM21]{Maes2021}
  Jeroen Maes and Fernando Muro, \emph{Derived homotopy algebras},
  arXiv:2106.14987 [math.AT], 2021.

  \bibitem[Mur20]{muro_2020_enhanced_finite_triangulated}
  Fernando Muro, \emph{Enhanced {{Finite Triangulated Categories}}}, Journal of
  the Institute of Mathematics of Jussieu (2020), 1--43 (en).

  \bibitem[Mur21]{Muro2021}
  \bysame, \emph{Derived universal {{Massey}} products}, arXiv:2109.01421 [math]
  to appear in Homology Homotopy Appl., 2021.

  \bibitem[Nom54]{Nomizu1954}
  Katsumi Nomizu, \emph{On the cohomology of compact homogeneous spaces of
    nilpotent lie groups}, The Annals of Mathematics \textbf{59} (1954), 531.

  \bibitem[Ret77]{retah_1977_massey_operations_lie}
  V.~S. Retah, \emph{The {{Massey}} operations in {{Lie}} superalgebras, and
    deformations of complex-analytic algebras}, Funkcional. Anal. i Prilo\v{z}en.
  \textbf{11} (1977), no.~4, 88--89. \MR{0508119}

  \bibitem[Ret78]{retah_1978_massey_operations_lie}
  \bysame, \emph{Massey operations in {{Lie}} superalgebras and differentials of
    the {{Quillen}} spectral sequence}, Funktsional. Anal. i Prilozhen.
  \textbf{12} (1978), no.~4, 91--92. \MR{515638}

  \bibitem[Sag10]{sagave_2010_dgalgebras_derived_algebras}
  Steffen Sagave, \emph{{{DG}}-algebras and derived
      {{{\emph{A}}}}{\textsubscript{{$\infty$}}}-algebras}, J. Reine Angew. Math.
  \textbf{639} (2010), 73--105.

\end{thebibliography}
